\documentclass[12pt]{article}

\usepackage{CJK,CJKnumb,amsmath}
\usepackage{amsfonts}
\usepackage{amsthm}
\usepackage{geometry}
\usepackage{mathrsfs}
\usepackage{amsmath}
\usepackage{amssymb}
\usepackage[percent]{overpic}
\usepackage{bm}
\usepackage{marginnote}
\usepackage[flushmargin]{footmisc}
\usepackage[all]{xy}
\usepackage{graphicx}
\usepackage{subfigure}
\usepackage{latexsym}
\usepackage[colorlinks, linkcolor=blue, anchorcolor=blue, citecolor=blue]{hyperref}
\usepackage{epigraph}
\usepackage{hyperref}
\usepackage{fancyhdr}
\usepackage{comment}

\usepackage{mathtools}
\mathtoolsset{showonlyrefs}

\numberwithin{equation}{section}

\usepackage{color}

\newtheorem{theorem}{Theorem}

\newtheorem{definition}{Definition}[section]
\newtheorem{proposition}{Proposition}[section]

\newtheorem{lemma}{Lemma}[section]
\theoremstyle{definition}
\newtheorem{remark}{Remark}

\newtheorem*{ack}{Acknowledgement}

\def\re{\operatorname{Re}}
\def\Im{\operatorname{Im}}

\def\dens{\operatorname{dens}}
\def\meas{\operatorname{meas}}
\def\diam{\operatorname{diam}}
\def\dist{\operatorname{dist}}

\def\c{\operatorname{\mathbb C}}

\def\j{\operatorname{\mathcal{J}}}

\def\p{\operatorname{\mathcal{P}}}

\def\nr{\operatorname{\mathcal{N}\mathcal{R}}}

\begin{document}
\title{Perturbations of exponential maps: Non-recurrent dynamics}
\author{Magnus Aspenberg and Weiwei Cui}
\date{}
\maketitle

\begin{abstract}
We study perturbations of non-recurrent parameters in the exponential family. It is shown that the set of such parameters has Lebesgue measure zero. This particularly implies that the set of escaping parameters has Lebesgue measure zero,  which complements a result of Qiu from 1994. Moreover, we show that non-recurrent parameters can be approximated by hyperbolic ones.


\end{abstract}

\section{Introduction and main results}
In this paper we study the exponential family
$$f_{\lambda}(z)=\lambda e^z\quad\text{ for }\quad\lambda\in\c\setminus\{0\}.$$
Regarded as the simplest transcendental functions, they have attracted a lot of attention in transcendental dynamics since 1980s (see \cite{bergweiler1} for an introduction to the field). Considerable efforts have been put to explore this family by Devaney and his coauthors (see \cite{devaney1, devaney8} for instance), Baker and Rippon \cite{baker2},  Rempe \cite{rempethesis}, etc. Currently we have a good understanding of these functions both in the dynamical and parameter spaces. Several challenging problems are, however, still open up to now. For instance, it is unknown whether the bifurcation locus of the exponential family has Lebesgue measure zero. Our paper can be viewed as a contribution to this problem.

We will focus on parameters $\lambda$ for which the only singular value $0$ is in the Julia set and the $\omega$-limit set of $0$ does not contain itself. Such parameters are called \emph{non-recurrent}. We first show that non-recurrent dynamics are rare in the Lebesgue sense.

\begin{theorem}\label{maintheorem1}
The set of non-recurrent parameters in the exponential family has zero Lebesgue measure.
\end{theorem}

We make several remarks on this result. Let $\j(f_{\lambda})$ be the Julia set of $f_{\lambda}$.
\begin{remark}~
\begin{itemize}
\item[$\bullet$] Bade\'nska considered parameters in the exponential family with bounded post-singular sets (called \emph{post-singularly bounded parameters}) and proved a similar result \cite[Theorem 1]{badenska1} for such parameters. This implies directly that such parameters are non-recurrent in the above sense. So our theorem gives a generalization of her result.

\item[$\bullet$] A parameter $\lambda$ is \emph{escaping} if $f^{n}_{\lambda}(0)\to\infty$ as $n\to\infty$. Qiu proved in 1994 that the set of escaping parameters has Hausdorff dimension two \cite{qiu2}. It was, however, unknown since then whether this set has Lebesgue measure zero (which is the motivation of the present paper). Our result confirms this and thus complements his result, since escaping parameters are also non-recurrent. However, escaping parameters may not be rare in the above sense for general families of transcendental functions. In \cite{qiu2} it is shown that the set of escaping parameters in the sine family $\{\lambda\sin(z):\lambda\in\c\setminus\{0\}\}$ has positive Lebesgue measure.

\item[$\bullet$] Non-recurrent parameters are considerably more general due to the non-compactness of the phase space. In the above settings, the singular value either has a bounded orbit or tends to $\infty$ under iterates. Non-recurrent parameters, however, can be post-singularly unbounded without being escaping. In other words, the singular orbit could possibly oscillate between some compact set and $\infty$. So in this sense our result is more general.

\item[$\bullet$] For non-recurrent parameters, we note the following difference between escaping and non-escaping ones. For non-escaping non-recurrent maps, the singular value belongs to the radial Julia set (and hence can go from small scales to large scales by univalent iterates); see \cite[Section 3]{rempe5} and \cite{rempe17} for a discussion and related results on this set. While for escaping parameters, the singular value does not lie in this set. (We are grateful to Lasse Rempe for making this comment.)

\end{itemize}
\end{remark}

\begin{remark}
In the study of exponential dynamics, one is often led to a comparison with the (complex) quadratic family $z\mapsto z^2+c$ which gives rise to the famous Mandelbrot set. For quadratic polynomials, non-recurrence is equivalent to requiring the post-critical set to be a hyperbolic set. However, for exponential maps, that the post-singular set is a hyperbolic set means that the map is post-singularly bounded. Thus, the notion of non-recurrence is weaker in the exponential setting.
\end{remark}

\smallskip
A parameter $\lambda$ is \emph{hyperbolic} if $f_{\lambda}$ has an attracting cycle. One of the main problems in the dynamics of exponential maps is whether hyperbolic exponential maps are dense in the parameter space (\emph{density of hyperbolicity}). This is still widely open. Equivalently, one asks if every non-hyperbolic map can be approximated by hyperbolic maps.  As an application of the argument used in the proof of Theorem \ref{maintheorem1}, we show that this is true for all non-recurrent parameters.

\begin{theorem}\label{maintheorem2}
Every non-recurrent exponential map can be approximated by hyperbolic maps.
\end{theorem}

The result was partially known for some non-recurrent parameters; see, for instance, Devaney \cite{devaney12}, Zhou and Li \cite{zhou-li} and Ye \cite{yezhuan1}. Indeed, rigidity implies that such parameters can be approximated by hyperbolic ones. For post-singularly finite exponential maps, rigidity was known; see Benini \cite{benini4}.  For escaping parameters this was proved by Rempe in \cite{rempe12}.

\begin{remark}
Dobbs proved a result for post-singularly bounded parameters \cite[Main Theorem]{dobbs1}, saying that such parameters are Lebesgue density points of the set of hyperbolic maps. It is plausible that this is also true for non-recurrent parameters.
\end{remark}

\smallskip

We have mainly focused on non-recurrent exponential maps and their relevant properties in this paper. The parameter space of the exponential family has a rich structure and has attracted a lot of attention. Without going into any further details, we refer to \cite{eremenko2, devaney15, qiu2, urbanski4, rempe3, schleicher12, benini3, benini4, levin6, bergweiler25} and references therein.

\section{Preliminaries}

In this section we give some preliminary results for non-recurrent exponential maps. These include uniform expansion on the post-singular set and the existence of a holomorphic motion along this orbit.

First we give some notations which will be used throughout this paper. We will use $D(a,r)$ for a Euclidean disk of radius $r$ centered at $a$. Let $\dist(z, A)$ be the Euclidean distance of a point $z$ to a set $A$. The term $Df_{\lambda}^{n}(z)$ will always mean the derivative of $f_{\lambda}^{n}$ at the point $z$ (i.e., the phase derivative). Sometimes we also use $(f^{n}_{\lambda})'(z)$ as the same meaning for $Df^{n}_{\lambda}(z)$. On the other hand, $\partial_{\lambda} f_{\lambda}(z)$ will mean the derivative in $\lambda$ (i.e., the parameter derivative). The complex plane is always denoted by $\c$. Moreover, $\meas A$ is the two-dimensional Lebesgue measure of a set $A\subset \c$ and $\dens(A,B)$ is the density of $A$ in $B$; i.e., $\dens(A, B)=\meas(A\cap B)/\meas(B)$. By $a\sim b$ we mean that there exists a constant $C>0$ such that $\frac{1}{C} b\leq a \leq C b$. 

\medskip
Let $f_{\lambda}$ be an exponential map. The \emph{post-singular set} of $f_{\lambda}$ is defined as
$$\p(f_{\lambda}):=\overline{\bigcup_{n\geq 0}f_{\lambda}^{n}(0)}.$$

Recall that $\lambda$ is called \emph{non-recurrent} if $0\in\j(f_{\lambda})$ and $0\not\in\omega(0)$. In other words, there exists $\Delta>0$ such that
\begin{equation}\label{dnr}
\p(f_{\lambda})\cap D(0,\Delta)=\{0\}.
\end{equation}

Note that non-recurrent exponential maps cannot have attracting nor parab-olic cycles. That Siegel disks do not exist was proved by Rempe and van Strien \cite[Corollary 2.10]{rempe5}. So we see that non-recurrent exponential maps have empty Fatou sets. 

\begin{definition}\label{def21}
$f_{\lambda}$ is said to be $\Delta$-non-recurrent if there exists $\Delta>0$ such that \eqref{dnr} holds. Equivalently, we say that such a parameter $\lambda$ is $\Delta$-non-recurrent.
\end{definition}

One of the ingredients in the proof of our theorem is the expansion along the post-singular set of non-recurrent maps. The following result was proved by Benini \cite[Corollary A]{benini3}.

\begin{lemma}\label{lem21}
Let $\lambda$ be a non-recurrent parameter. Then there exist $\widetilde{N}\in\mathbb{N}$ and $\tilde{\gamma}>1$ such that for any $k\geq \widetilde{N}$ and for any $z\in\p(f_{\lambda})$ we have
$$|Df_{\lambda}^{k}(z)|>\tilde{\gamma}.$$
\end{lemma}

Let $\lambda$ be a $\Delta$-non-recurrent parameter. Let $\widetilde{N}, \tilde{\gamma}$ be as in the above lemma. For any integer $n$ there exist $k$ and $j\in\mathbb{N}$ such that $n=k\widetilde{N}+j$, where $k\geq 0$ and $0\leq j<\widetilde{N}$. So for any $z\in\p(f_{\lambda})$ we see from the above lemma that
\begin{equation}\label{efn}
\begin{aligned}
\left|Df_{\lambda}^{n}(z)\right|=\left|Df_{\lambda}^{k\widetilde{N}+j}(z)\right|&=\left|Df_{\lambda}^{k\widetilde{N}}(f_{\lambda}^{j}(z))\right|\left|Df_{\lambda}^{j}(z)\right|\\
&=\prod_{m=0}^{k-1}\left|Df_{\lambda}^{\widetilde{N}}(f_{\lambda}^{m\widetilde{N}+j}(z))\right|\left|Df_{\lambda}^j(z)\right|\\
&\geq \tilde{\gamma}^{k}|Df_{\lambda}^j(z)|=\tilde{\gamma}^{k}\prod_{i=1}^{j}|f_{\lambda}^i(z)|.
\end{aligned}
\end{equation}
Since $f_{\lambda}$ is $\Delta$-non-recurrent, we have $\prod_{i=1}^{j}|f_{\lambda}^i(z)|\geq \Delta^{j}$. Put
$$\tilde{C}=\min\left\{1,\,\left(\frac{\Delta}{\tilde{\gamma}}\right)^{\widetilde{N}}\right\}.$$
We obtain from \eqref{efn} that
\begin{equation}\label{uefn}
\begin{aligned}
\left|Df_{\lambda}^{n}(z)\right|\geq \tilde{C}\tilde{\gamma}^k=\tilde{C}\tilde{\gamma}^{(n-j)/\widetilde{N}}\geq\tilde{C}\tilde{\gamma}^{(n-\widetilde{N})/\widetilde{N}}=\frac{\tilde{C}}{\tilde{\gamma}}\tilde{\gamma}^{n/\widetilde{N}}=:C_1\gamma_{1}^{n},
\end{aligned}
\end{equation}
where $\gamma_1>1$ and $C_1>0$. 

More generally, a parameter $\lambda$ is said to be \emph{summable}, if $0\in\j(f_{\lambda})$ and
$$\sum_{n=0}^{\infty}\frac{1}{\left|Df_{\lambda}^{n}(0)\right|}<\infty.$$
See \cite{urbanski4}. Therefore, for non-recurrent parameter $\lambda$, by taking $z=0$ in \eqref{uefn} we see that
$$\sum_{n=0}^{\infty}\frac{1}{|Df_{\lambda}^n(0)|}\leq\sum_{n=0}^{\infty}\frac{1}{C_1 \gamma_{1}^{n}}<\infty.$$
In other words, we have
\begin{lemma}\label{sum1}
Non-recurrent parameters are summable.
\end{lemma}

\eqref{uefn} tells that $f_{\lambda}$ is uniformly expanding along $\p(f_{\lambda})$. (We say that a map $f$ is \emph{uniformly expanding} on a closed set $E\subset\c$ (which could be unbounded), if there exist $C>0$ and $\gamma>1$ such that $|Df^{n}(z)|\geq C \gamma^n$ for all $n$ and for all $z\in E$. If $E$ is also bounded, then such a set is often called a hyperbolic set.) 

\medskip
Let $\lambda_0$ be non-recurrent. Then there exists $\eta>0$ such that $|Df^{n}_{\lambda_0}(z)|>C_1\gamma_{1}^n$ for all $z\in \mathcal{N}_0$, where $\mathcal{N}_0:=D(\p(f_{\lambda_0}), 100\eta)$ is a neighborhood of $\p(f_{\lambda_0})$. By continuity, parameters close to $\lambda_0$ also enjoy this property (but with possibly a slightly smaller exponent). So we have the following.

\begin{lemma}\label{exue}
Let $\lambda_0$ be a non-recurrent parameter. Then there exist $r_0$ sufficiently small and constants $C>0$, $\gamma>1$ such that for all $r\leq r_0$ and for all $\lambda\in D(\lambda_0,r)$, whenever $f_{\lambda}^j(z)\in\mathcal{N}:=D(\p(f_{\lambda_0}), 50\eta)$ for all $j\leq n$ we have
$$|Df_{\lambda}^n(z)|\geq C\gamma^n.$$
\end{lemma}

The uniform expansion also implies the existence of a holomorphic motion over $\p(f_{\lambda_0})$ in a neighborhood of the parameter $\lambda_0$. 

\begin{lemma}[Holomorphic motion]\label{hm}
Let $\lambda_0$ be non-recurrent. Then for sufficiently small $r>0$ there exists a holomorphic motion
$$h: D(\lambda_0, r)\times \p(f_{\lambda_0})\to \c$$
such that $h_{\lambda_0}$ is identity, $h$ is holomorphic in $\lambda$ and injective in $z$. For all $\lambda\in D(\lambda_0,r)$, we have
$$h_{\lambda}\circ f_{\lambda_0}=f_{\lambda}\circ h_{\lambda}$$
for $z\in \p(f_{\lambda_0})$. Moreover, there exists a constant $\alpha>0$ such that $|h_{\lambda}(z)-z|\leq \alpha|\lambda-\lambda_0|$ for any $z\in\p(f_{\lambda_0})$ and $\lambda\in D(\lambda_0, r)$.
\end{lemma}

\begin{proof}
Let $\mathcal{N}_0$ be the neighborhood of $\p(f_{\lambda_0})$ as above. Then by the above Lemma \ref{exue}, we see that $f_{\lambda}$ is uniformly expanding in $\mathcal{N}$, where $\lambda\in D(\lambda_0, r)$. By taking $r_1>0$ sufficiently small such that $D(z,r_1)\subset\mathcal{N}$ for all $z\in\p(f_{\lambda_0})$ and using the expansion of $f_{\lambda_0}$ in $\mathcal{N}$, we can choose a sufficiently small $r$ such that the following holds: For all $\lambda\in D(\lambda_0, r)$ and all $z\in\mathcal{N}$ we have $D(f_{\lambda_0}(z), r_1)\subset f_{\lambda}(D(z,r_1))$. This means that the pullback $f_{\lambda}^{-n}(D(z,r_1))$ of $D(z,r_1)$ under $f_{\lambda}^{n}$ shrinks exponentially, where the inverse branch of $f_{\lambda}$ is taken suitably. Therefore, we can find for each $z\in\p(f_{\lambda_0})$ a unique point $z_{\lambda}$ such that $f_{\lambda}^{n}(z_{\lambda})\in D(f_{\lambda_0}^{n}(z), r_1)$ for all $n$. Put $h_{\lambda}(z):=z_{\lambda}$. It follows from the construction of $h_{\lambda}$ that
$$h_{\lambda}\circ f_{\lambda_0}=f_{\lambda}\circ h_{\lambda}$$
for $z\in \p(f_{\lambda_0})$. Moreover, $h_{\lambda}$ is injective in $z$. That $h$ is holomorphic in $\lambda$ can be seen by looking at $z_{\lambda}^{n}:=f_{\lambda}^{-n}(f_{\lambda_0}^{n}(z))$. By the above argument $z_{\lambda}^{n}$ converges to $z_{\lambda}$ uniformly as $n\to\infty$. So $h$ is indeed holomorphic in $\lambda$. To sum up, $h(\lambda, z)$ thus defined is a holomorphic motion.

The last statement of the lemma follows from the above construction. Observe first that 
\begin{equation}\label{onestep}
|w_{\lambda}^{1}-w|=|\log(\lambda/\lambda_0)|
\end{equation}
 holds for any $w\in\p(f_{\lambda_0})$. Let $z$ and $z_{\lambda}$ be as above. So we have
$$
\begin{aligned}
|h_{\lambda}(z)-z|=|z_{\lambda}-z|&\leq|z_{\lambda}^{1}-z|+\sum_{n=1}^{\infty}|z_{\lambda}^{n+1}-z_{\lambda}^{n}|\\
&\leq |z_{\lambda}^{1}-z|+\sum_{n=1}^{\infty}\frac{1}{|Df_{\lambda}^{n}(z)|}\left|f_{\lambda_0}^{n}(z)-f_{\lambda}^{-1}(f_{\lambda_0}^{n+1}(z))\right|\\
&\leq \left(1+\sum_{n=1}^{\infty}\frac{1}{C\gamma^n}\right)\left|\log\frac{\lambda}{\lambda_0}\right|\\
&\leq \alpha \left|{\lambda}-{\lambda_0}\right|.
\end{aligned}
$$
In the second inequality of the above estimate, we have taken the inverse branch of $f_{\lambda}$ sending $f_{\lambda_0}^{n}(z)$ to $f_{\lambda_0}^{n+1}(z)$. In the third inequality, we used Lemma \ref{exue} and  \eqref{onestep}. Since $r>0$ is taken sufficiently small, so we have the local expansion: $|\log \lambda/\lambda_0|\sim |\lambda-\lambda_0|$. This, together with the fact that $\gamma>1$, gives the constant $\alpha$.
\end{proof}


The holomorphic motion and the expansion together imply that any two points in $\mathcal{N}$ will repel each other for all maps in $D(\lambda_0, r)$. Our analysis later tells that more information can be drawn during this process. For this purpose, we will specify a number $\delta>0$ such that the set of all points $z$ satisfying $\dist(z, h_{\lambda}(\p(f_{\lambda_0})))<10\delta$ is contained in $\mathcal{N}$ for all $\lambda\in D(\lambda_0, r)$. For instance, one can choose $\delta<2\eta$ so that Lemmas \ref{exue} and \ref{hm} hold.

Using the above lemma, we have the following result.

\begin{lemma}\label{exdis}
Let $\lambda_0$ be non-recurrent. Then there exist $r>0$ sufficiently small, constants $C$ and $\gamma>1$ such that for any $\lambda\in D(\lambda_0, r)$ we have
$$|f_{\lambda}^n(z)-f_{\lambda}^{n}(w)|\geq C \gamma^n |z-w|,$$
if $f_{\lambda}^{j}(z), f_{\lambda}^{j}(w)\in\mathcal{N}$ and $|f_{\lambda}^{j}(z)-f_{\lambda}^{j}(w)|\leq \delta$ for all $j\leq n$.
\end{lemma}

\section{Phase-parameter relation and distortions}

In this section we prove some distortion results by using the expansion property in the previous section. This, together with a transversality result of Urba\'nski and Zdunik \cite{urbanski4}, gives us sufficiently good control over the distortion in the parameter space.

\subsection{Distortions}
Define 
\begin{equation}\label{xifun}
\xi_{n}(\lambda):=f_{\lambda}^{n}(0).
\end{equation}
Now we assume that $\lambda_0$ is a $\Delta_0$-non-recurrent parameter for some $\Delta_0>0$. Let $D(\lambda_0, r)$ be a disk of radius $r$ around $\lambda_0$ for some sufficiently small $r>0$. For simplicity we put $\Lambda_0:=\p(f_{\lambda_0})$. We also define 
$$\mu_n(\lambda)=h_{\lambda}(\xi_n(\lambda_0)),$$
where $h_{\lambda}$ is the holomorphic motion in Lemma \ref{hm}. 

\medskip
We shall use the following simple lemma.
\begin{lemma}\label{simplemma}
Let $a_k\in\c$ for $1\leq k\leq n$. Then
$$\left|\prod_{k=1}^{n} (1+a_k) - 1\right|\leq \exp\left(\sum_{k=1}^{n}|a_k| \right) - 1.$$
\end{lemma}

\begin{lemma}\label{maindistortion}
Let $\lambda_0$ be non-recurrent. For any $\varepsilon>0$ there exist $\delta>0$ and $r>0$ sufficiently small such that for any $\lambda\in D(\lambda_0,r)$, if $\xi_{j}(\lambda),\,\mu_{j}(\lambda)\in\mathcal{N}$ and $|\xi_{j}(\lambda)-\mu_{j}(\lambda)|\leq \delta$ for all $0\leq j\leq n$, then
$$
\left|\frac{(f_{\lambda}^{n})'(\mu_0(\lambda))}{(f_{\lambda}^{n})'(\xi_0(\lambda))}-1\right|<\varepsilon.
$$
\end{lemma}

\begin{proof}
By the chain rule,
\begin{equation}\label{1steq}
\begin{aligned}
\left|\frac{(f_{\lambda}^{n})'(\mu_0(\lambda))}{(f_{\lambda}^{n})'(\xi_0(\lambda))}-1\right|&=\left|\prod_{j=0}^{n-1}\frac{f'_{\lambda}(f_{\lambda}^{j}(\mu_0(\lambda)))}{f'_{\lambda}(f_{\lambda}^{j}(\xi_0(\lambda)))} - 1 \right|\\
&=\left|\prod_{j=0}^{n-1} \left(1+ \frac{f'_{\lambda}(f_{\lambda}^{j}(\mu_0(\lambda)))}{f'_{\lambda}(f_{\lambda}^{j}(\xi_0(\lambda)))}-1\right) - 1 \right|.
\end{aligned}
\end{equation}
We put
$$v_j:=\frac{f'_{\lambda}(f_{\lambda}^{j}(\mu_0(\lambda)))}{f'_{\lambda}(f_{\lambda}^{j}(\xi_0(\lambda)))}-1=\frac{f'_{\lambda}(\mu_j(\lambda))}{f'_{\lambda}(\xi_j(\lambda))} - 1.$$

By Lemma \ref{simplemma}, to show that \eqref{1steq} can be made arbitrarily small, it suffices to prove that $\sum_{j=0}^{n-1}|v_j|$ is sufficiently small. Note first that $f'_{\lambda}(\xi_j(\lambda))=f_{\lambda}(\xi_j(\lambda)))=\xi_{j+1}(\lambda)$ and $f'_{\lambda}(\mu_j(\lambda))=f_{\lambda}(\mu_j(\lambda))=\mu_{j+1}(\lambda)$. So,
$$
\sum_{j=0}^{n-1}|v_j|=\sum_{j=0}^{n-1}\left|\frac{f'_{\lambda}(\mu_j(\lambda))}{f'_{\lambda}(\xi_j(\lambda))} - 1 \right|=\sum_{j=0}^{n-1}\left|\frac{\mu_{j+1}(\lambda)}{\xi_{j+1}(\lambda)} - 1\right|
$$
By the $\Delta_0$-non-recurrence of $f_{\lambda_0}$ and Lemma \ref{exue}, we have
$$
\begin{aligned}
\sum_{j=0}^{n-1}|v_j|&\leq \frac{2}{\Delta_0}\sum_{j=0}^{n-1}\left|\mu_{j+1}(\lambda)-\xi_{j+1}(\lambda) \right|\\
&\leq \frac{2}{\Delta_0}\sum_{j=0}^{n-1}\frac{1}{C}\gamma^{j+1-n} \left|\mu_{n}(\lambda)-\xi_{n}(\lambda)\right|\\
&\leq \frac{2}{\Delta_0}C_1(\gamma)\delta.
\end{aligned}
$$
Here $C_1(\gamma)>0$ is a constant depending on $\gamma$. Now since $\delta$ can be made arbitrarily small we see that \eqref{1steq} is also small. This completes the proof of the lemma.
\end{proof}

\begin{lemma}\label{nearby}
Let $\lambda_0$ be non-recurrent. For any $\varepsilon>0$ there exist $\delta>0$ and $r>0$ sufficiently small such that for any $\lambda_1,\,\lambda_2\in D(\lambda_0,r)$, if $\xi_{j}(\lambda_i),\,\mu_{j}(\lambda_i)\in\mathcal{N}$ and $|\xi_{j}(\lambda_i)-\mu_{j}(\lambda_i)|\leq \delta$ for $i=1,2$ and all $0\leq j\leq n$, then
$$
\left|\frac{(f_{\lambda_1}^{n})'(0)}{(f_{\lambda_2}^{n})'(0)}-1\right|<\varepsilon.
$$
\end{lemma}

\begin{proof}
The proof is similar to the proof of the above lemma. First we note that by Lemma \ref{maindistortion} it suffices to show the following holds:
\begin{equation}\label{aes}
\left|\frac{Df_{\lambda_1}^n (\mu_0(\lambda_1))}{Df_{\lambda_2}^n (\mu_0(\lambda_2))} - 1 \right|=\left|\prod_{j=0}^{n-1}\frac{f'_{\lambda_1}(\mu_j(\lambda_1))}{f'_{\lambda_2}(\mu_j(\lambda_2))} -1 \right|=\left|\prod_{j=0}^{n-1}\frac{\mu_{j+1}(\lambda_1)}{\mu_{j+1}(\lambda_2)} -1 \right|<\varepsilon.
\end{equation}
With
$$w_j=\frac{\mu_{j+1}(\lambda_1)}{\mu_{j+1}(\lambda_2)} -1,$$
it is enough to show, by Lemma \ref{simplemma}, that $\sum_{j=0}^{n-1}|w_j|$ is sufficiently small. By the $\Delta_0$-non-recurrence of $f_{\lambda_0}$ and Lemma \ref{hm}, we have
$$
\begin{aligned}
|w_j|&\leq \frac{2}{\Delta_0}\left|\mu_{j+1}(\lambda_1)-\mu_{j+1}(\lambda_2)\right|\\
&=\frac{2}{\Delta_0}\left|h_{\lambda_1}\left(\xi_{j+1}(\lambda_0)\right)-h_{\lambda_2}(\xi_{j+1}\left(\lambda_0)\right)\right|\\
&\leq \frac{2}{\Delta_0}\left(\left|h_{\lambda_1}\left(\xi_{j+1}(\lambda_0)\right)-h_{\lambda_0}(\xi_{j+1}(\lambda_0)\right|+\left|h_{\lambda_2}(\xi_{j+1}\left(\lambda_0)\right)-h_{\lambda_0}(\xi_{j+1}(\lambda_0)\right|\right)\\
&\leq\frac{2\alpha}{\Delta_0}\left(|\lambda_1-\lambda_0|+|\lambda_2-\lambda_0|\right).
\end{aligned}
$$
So we have
\begin{equation}\label{this2}
\sum_{j=0}^{n-1}|w_j|\leq \frac{2\alpha}{\Delta_0} n\left(|\lambda_1-\lambda_0|+|\lambda_2-\lambda_0|\right).
\end{equation}

Now it follows from Lemma \ref{exdis},  for $\lambda\in D(\lambda_0, r)$ we also have
\begin{equation}\label{relaest}
\begin{aligned}
\delta\geq |\xi_n(\lambda)-\mu_n(\lambda)|&=\left|f_{\lambda}^{n}(\xi_0(\lambda))-f_{\lambda}^{n}(\mu_0(\lambda))\right|\\
&\geq C\gamma^n \left|\xi_0(\lambda)-\mu_0(\lambda)\right|\\
&=C\gamma^n |h_{\lambda}(0)|,
\end{aligned}
\end{equation}
where $h_{\lambda}$ is the holomorphic motion and thus has the local expansion
\begin{equation}\label{lehm}
h_{\lambda}(0)=a_K(\lambda-\lambda_0)^{K}+\mathcal{O}\left((\lambda-\lambda_0)^{K+1}\right)
\end{equation}
for some constant $a_K$ and $K\in\mathbb{N}$. Combining \eqref{relaest} with \eqref{lehm}  we get an estimate for the time $n$: For some constant $C'>0$ depending only on $\delta$ but not on $n$,
\begin{equation}\label{estofn}
n\log\gamma\leq -C'\log |\lambda-\lambda_0|.
\end{equation}

Putting \eqref{estofn} into \eqref{this2}, we see that
\[
\sum_{j=0}^{n-1}|w_j|=\mathcal{O}\left(|\lambda_1-\lambda_0|\log|\lambda_1-\lambda_0|+|\lambda_2-\lambda_0|\log|\lambda_2-\lambda_0| \right),
\]
which can be made sufficiently close to $0$ since $r$ can be taken sufficiently small  and $|\lambda_i-\lambda_0|<r$ for $i=1, 2$. Thus \eqref{aes} holds. This finishes the proof of the lemma.
\end{proof}

\subsection{Phase-parameter relations: transversality}

In the following we need to compare the phase and parameter derivatives. We start with some calculations. Recall from \eqref{xifun} that
$$
\xi_n(\lambda)=f_{\lambda}^{n}(0).
$$
Now we write
$$\xi_n(\lambda)=f(\lambda, \xi_{n-1}(\lambda)),$$
$$\alpha_j(\lambda)=f'_{\lambda}(\xi_{j}(\lambda))$$
and
$$\rho_j(\lambda)=\partial_{\lambda}f(\lambda, \xi_j(\lambda)).$$
Then for all $n\geq 0$, we have, by induction,
$$
\begin{aligned}
\xi'_{n+1}(\lambda)&=f'_{\lambda}(\xi_n(\lambda))\cdot\xi'_{n}(\lambda)+\partial_{\lambda}f(\lambda,\xi_{n}(\lambda))\\
&=\alpha_n(\lambda)\cdot\xi'_{n}(\lambda)+\rho_n(\lambda)\\
&=\alpha_n(\lambda)\left( \alpha_{n-1}(\lambda)\cdot\xi'_{n-1}(\lambda)+\rho_{n-1}(\lambda)\right)+\rho_n(\lambda)\\
&=\prod_{j=1}^{n}\alpha_j(\lambda)\cdot\xi'_1(\lambda)+\sum_{j=1}^{n-1}\left(\prod_{k=j+1}^{n}\alpha_k(\lambda)\right)\rho_j(\lambda)+\rho_{n}(\lambda)\\
&=\prod_{j=1}^{n}\alpha_j(\lambda)\left(\xi'_1(\lambda)+\sum_{j=1}^{n}\frac{1}{\lambda^j\prod_{k=0}^{j-1}e^{\xi_k(\lambda)}} \right)\\
&=\prod_{j=1}^{n}\alpha_j(\lambda)\left( \xi'_1(\lambda)+\sum_{j=1}^{n}\prod_{k=1}^{j}\frac{1}{\xi_k(\lambda)}\right).
\end{aligned}
$$
Since $\xi'_{1}(\lambda)=1$ and
$$
\prod_{j=1}^{n}\alpha_j(\lambda)=\prod_{j=1}^{n}f'_{\lambda}(\xi_j(\lambda))=\left(f_{\lambda}^{n}\right)'(\lambda),
$$
we obtain immediately that
\begin{equation}\label{phase-para}
\begin{aligned}
\xi'_{n+1}(\lambda)&=\left(f_{\lambda}^{n}\right)'(\lambda)\left( 1+\sum_{j=1}^{n}\prod_{k=1}^{j}\frac{1}{\xi_k(\lambda)}\right)\\&=\frac{1}{\lambda}\left(f_{\lambda}^{n+1}\right)'(0)\left( 1+\sum_{j=1}^{n}\frac{1}{(f_{\lambda}^{j})'(0)}\right)\\
&=\frac{1}{\lambda}\left(f_{\lambda}^{n+1}\right)'(0)\sum_{j=0}^{n}\frac{1}{(f_{\lambda}^{j})'(0)}.
\end{aligned}
\end{equation}
We used that $(f_{\lambda}^{j})'(0)=\prod_{k=1}^{j}f_{\lambda}^{k}(0)=\prod_{k=1}^{j}\xi_k(\lambda)$ and that $f_{\lambda}^{0}$ is the identity map. The sum appearing in the last is actually a truncated term (after taking absolute value) occurring in the summability condition. Therefore, for summable $\lambda$, the limit
$$
L_{\lambda}:=\lim_{n\to\infty}\sum_{j=0}^{n}\frac{1}{(f_{\lambda}^{j})'(0)}<\infty
$$
exists. Actually, one can say more for non-recurrent parameters.

\begin{lemma}\label{Levin}
Let $\lambda$ be a non-recurrent parameter. Then $L_{\lambda}$ exists and is not equal to $0$ nor $\infty$.
\end{lemma}

\begin{proof}
By Lemma \ref{sum1}, $\lambda$ is summable. It follows then from the work of Urba\'nski and Zdunik \cite[Sections 4 and 5]{urbanski4} that $L_{\lambda}\neq 0,\infty$ for the non-recurrent parameter $\lambda$.
\end{proof}

By using Lemma \ref{Levin} and \eqref{phase-para}, we have a relation between phase and parameter derivatives in the following sense: For  non-recurrent $\lambda_0$, there exists $L_{\lambda_0}(\neq 0,\infty)$ such that
$$
\lim_{n\to\infty}\frac{\xi'_{n}(\lambda_0)}{\left(f_{\lambda_0}^{n}\right)'(0)}=\frac{L_{\lambda_0}}{\lambda_0}.
$$
This relation is persistent for parameters $\lambda$ which are sufficiently close to $\lambda_0$ as long as expansion on $\p(f_{\lambda_0})$ is ensured and the singular orbit of $f_{\lambda}$ stays close to $\p(f_{\lambda_0})$. More precisely, we have:

\begin{lemma}\label{trans}
Let $\lambda_0$ be non-recurrent. Then for any $q\in (0,1)$ there exist $N_1>0$ and $r>0$ sufficiently small such that the following holds. For any $\lambda\in D(\lambda_0,r)$, if $|\xi_{j}(\lambda)-\xi_{j}(\lambda_0)|\leq\delta$ for some $\delta=\delta(r)$ and for all $j\leq n$ with $n\geq N_1$, we have
$$
\left| \frac{\xi'_{n}(\lambda)}{\left(f_{\lambda}^{n}\right)'(0)}-\frac{L_{\lambda_0}}{\lambda_0}\right|\leq q\left|\frac{L_{\lambda_0}}{\lambda_0}\right|.
$$
\end{lemma}

\begin{proof}
By \eqref{phase-para} we have
$$\frac{\xi'_{n}(\lambda)}{\left(f_{\lambda}^{n}\right)'(0)}=\frac{1}{\lambda}\sum_{j=0}^{n-1}\frac{1}{(f_{\lambda}^j)'(0)}.$$
By taking $N_1$ sufficiently large so that $N_1>\widetilde{N}$ and using \eqref{uefn} and Lemma \ref{exue} we can have
$$\frac{1}{|\lambda|}\sum_{j=N_1+1}^{\infty}\frac{1}{|(f_{\lambda}^j)'(0)|}\leq \frac{q}{2}\left|\frac{L_{\lambda_0}}{\lambda_0}\right|.$$
On the other hand, by continuity one can find some $r>0$ such that if $\lambda\in D(\lambda_0, r)$,
$$\left|\frac{1}{\lambda}\sum_{j=0}^{N_1}\frac{1}{(f_{\lambda}^j)'(0)} -\frac{L_{\lambda_0}}{\lambda_0} \right|\leq \frac{q}{2}\left|\frac{L_{\lambda_0}}{\lambda_0}\right|.$$
Combining all the above we have
$$
\begin{aligned}
\left| \frac{\xi'_{n}(\lambda)}{\left(f_{\lambda}^{n}\right)'(0)}-\frac{L_{\lambda_0}}{\lambda_0}\right|&\leq \left|\frac{1}{\lambda}\sum_{j=0}^{N_1}\frac{1}{(f_{\lambda}^j)'(0)} -\frac{L_{\lambda_0}}{\lambda_0} \right|+\frac{1}{|\lambda|}\sum_{j=N_1+1}^{\infty}\frac{1}{|(f_{\lambda}^j)'(0)|}\\
&\leq q\left|\frac{L_{\lambda_0}}{\lambda_0}\right|.
\end{aligned}
$$
\end{proof}

Lemma \ref{trans} and Lemma \ref{nearby} together imply that the function $\xi_n$ is an almost affine map.

\begin{lemma}[Strong distortion]\label{stron}
Let $\lambda_0$ be non-recurrent. For any $\varepsilon>0$ there exist $\delta>0$ and $r>0$ sufficiently small such that for any $\lambda_1,\,\lambda_2\in D(\lambda_0,r)$, if $|\xi_{j}(\lambda_i)-\mu_{j}(\lambda_i)|\leq \delta$ for $i=1,2$ and all $0\leq j\leq n$ with $n\geq N_1$, then
$$
\left|\frac{\xi_j'(\lambda_1)}{\xi_j'(\lambda_2)}-1 \right|\leq \varepsilon.
$$
\end{lemma}

With this lemma we get that
\begin{equation}\label{affine}
\left|\xi_{n}(\lambda_1)-\xi_{n}(\lambda_2)\right|\sim |Df_{\lambda_0}^{n-1}(0)||\lambda_1 - \lambda_2|
\end{equation}
as long as $\xi_j(\lambda_1)$ and $\xi_j(\lambda_2)$ stay close in a neighborhood of the singular orbit of $f_{\lambda_0}$ for all $j\leq n$. As $|Df^{n}_{\lambda_0}(0)|$ grows (at least) exponentially we obtain immediately the following result.

\begin{lemma}[Large scale]\label{scale}
Let $\lambda_0$ be non-recurrent. There exists a number $S>0$ such that for all $r>0$ sufficiently small, there is an integer $n$ so that $D(\xi_n(\lambda_0), S/4)\subset\xi_{n}(D(\lambda_0, r))\subset\mathcal{N}$ and has diameter at least $S$.
\end{lemma}

\begin{proof}
Put $S=\delta/2$ with $\delta$ in the above lemma. If $\xi_{n}(D(\lambda_0, r))$ never reached the large scale, i.e.,
$$\diam \xi_{n}(D(\lambda_0, r)) < S$$
for all $n\geq 0$, then we have \eqref{affine} satisfied for all $n$. So we see that, for some constant $C>0$,
$$S>\diam \xi_{n}(D(\lambda_0, r)) \geq C \gamma^{n-1} r.$$
But this is impossible since $\gamma>1$ and $n$ can be taken sufficiently large. Hence, $\xi_{n}(D(\lambda_0, r))\supset D(\xi_n(\lambda_0), S/4)$ follows directly from Lemma \ref{stron}.
\end{proof}

\section{Non-recurrent dynamics are rare}

It follows from Lemma \ref{scale} that a small disk around a starting non-recurrent parameter will finally grow to a definite size under the function $\xi_{n}$ with bounded distortion. As long as this happens, any compact set will finally be covered within a few more iterates. For our purpose, we shall show that a definite portion of parameters in the disk $D(\lambda_0, r)$ are not $\Delta$-non-recurrent (i.e., those $\lambda$ for which $f^{n}_{\lambda}(0)$ belongs to $D(0,\Delta)$). 

Let $\nr$ and $\nr_{\Delta}$ be the set of non-recurrent and $\Delta$-non-recurrent parameters in the exponential family, respectively; see Definition \ref{def21}. We prove the following result, which implies directly Theorem \ref{maintheorem1}, since $\nr=\cup_{\Delta>0}\nr_{\Delta}.$

\begin{proposition}\label{props}
Let $\lambda_0$ be non-recurrent. For any $\Delta>0$ and for all sufficiently small $r>0$, we have
$$\dens\left(\nr_{\Delta}, D(\lambda_0, r)\right)=\frac{\meas\left(D(\lambda_0, r)\cap\nr_{\Delta} \right)}{\meas D(\lambda_0, r)}<1.$$
\end{proposition}

We also need the following general result of Baker \cite[Lemma 2.2]{baker15}.

\begin{lemma}\label{blowup}
Let $f$ be a transcendental entire function. Let $U$ be a neighbourhood of $z\in\j(f)$. Then for any compact set $K$ not containing an exceptional point of $f$ there exists $n(K)\in\mathbb{N}$ such that $f^{n}(U)\supset K$ for all $n\geq n(K)$.
\end{lemma}

An exceptional point is a point with finite backward orbit. Note that a transcendental entire function has at most one exceptional point. In our case, exponential functions have $0$ as an exceptional point. 

\medskip
To begin with the proof, let us put $D:=D(\lambda_0, r)$. Then by Lemma \ref{scale} for all $r>0$ sufficiently small one can find $n$ such that $\diam\xi_n(D)\geq S$ and
$$\xi_{n}(D)\supset D(\xi_{n}(\lambda_0), S/4)=:V.$$

Now we consider the annulus $A:=A(\Delta/4, \Delta)$, centered at the origin with inner and outer radii $\Delta/4$ and $\Delta$ respectively. The closure of $A$ is denoted by $\overline{A}$. Then by the above Lemma \ref{blowup}, there exists $N$ such that $f_{\lambda_0}^{N}(V)\supset \overline{A}$. Put $A'=A(\Delta/2, 3\Delta/4)$. Since $N$ is a finite number depending only on $f_{\lambda_0}$ and $V$, we have, by decreasing $r$ if necessary, that for all $\lambda\in \xi_{n}^{-1}(V)$,
$$f_{\lambda}^{N}(V)\supset A'.$$
With $V'=\{z\in V: f_{\lambda}^{N}(z)\in A'\}$, one can find  a positive number $c>0$ such that
$$\meas V'\geq c\meas V.$$

Let $\widehat{D}$ be the set of parameters in $D$ which are mapped into $V'$ by the map $\xi_n$. Then by Lemma \ref{stron}, there exists a constant $c'>0$ depending on $c$ such that
$$\meas \widehat{D}\geq c'\meas D.$$
Therefore, there is a definite portion of parameters $\lambda$ in $D$ for which $\xi_{n+N}(\lambda)\in A'$. In other words, such $\lambda$'s are not $\Delta$-non-recurrent. This ends the proof of Proposition \ref{props}.

\section{Non-recurrence and hyperbolicity}

In the section we prove approximation of non-recurrent parameters by hyperbolic ones. One of the key ingredients is to control derivatives along the  singular orbit up to at least time $n$, which is the time reaching the large scale.

We use some notations from the previous section. That is, we use $D:=D(\lambda_0, r)$ and $V=D(\xi_{n}(\lambda_0), S/4)$. Let $D'\subset D$ be such that $\xi_n(D')=V$. Let $n$ be as in Lemma \ref{scale}. Then by this lemma, $\xi_{n}(D)$ has diameter at least $S$ and contains $V$. Moreover, it is almost round due to the strong distortion ensured by Lemma \ref{stron}. Now we consider the following disk
$$D_1:=D(z_1, 1)$$
such that
$$
\Im(z_1)=-i\arg(\lambda_0)\quad\text{and}\quad\re(z_1)\geq M,
$$
where $\arg(\lambda_0)\in [0,2\pi)$ and $M>0$ is some sufficiently large number. Then similarly as above, Lemma \ref{blowup} ensures the existence of a number $N$ such that $f_{\lambda_0}^{N}(V)\supset D_1$.

Now we see that $f_{\lambda_0}(D_1)$ is a large set which intersects with the horizontal line $\mathcal{L}:=\{z: \Im(z)=-i\arg(\lambda_0)\}$ knowing that $M$ is chosen to be large at the beginning. So we can choose a disk
$$D_2:=D(z_2, 1)\subset f_{\lambda_0}(D_1)$$
such that $z_2\in \mathcal{L}$ and $\re z_2\geq e^{M/2}>\re z_1$.
Continuing in a similar way, we obtain a sequence of points $z_n$ such that $z_n\in \mathcal{L}$ and $\re z_n>\exp(\re z_{n-1}/2)$. Moreover, with $D_n:=D(z_n,1)$ we have $f_{\lambda_0}(D_{n})\supset D_{n+1}$. By choosing $z_n$ in this way, we can actually make sure that $\re z_n$ grows exponentially.

 With 
$$x_0:=\max_{j\leq n}\left\{\re(z):\, z\in\zeta_j(D)\right\}+2S,$$
there exists a smallest integer $p$ such that
$$
\re(z_p)\geq x_0.
$$
Then we see that $f_{\lambda_0}(D_p)$ will contain a disk $D(z_{p+1}, 3/2)$ satisfying 
\begin{equation}\label{zp+1}
z_{p+1}\in\mathcal{L}+\pi i\quad\text{and}\quad \re(z_{p+1})>\exp(\re(z_p)/2).
\end{equation}
Here $\mathcal{L}+\pi i$ is the translation of $\mathcal{L}$ by $\pi i$.

The above construction tells that $f_{\lambda_0}^{N+p}(V)\supset D(z_{p+1},3/2)$. By decreasing $r$ (if necessary) we can make sure that for all $\lambda\in D'$,
$$
f_{\lambda}^{N+p}(V)\supset D_{p+1}=D(z_{p+1}, 1).
$$

 Let $\widehat{V}$ be the component of  $f_{\lambda}^{-(N+p)}(D_{p+1})$ that is contained in $V$ and let $\widehat{D}$ be the set of points in $D'$ which are mapped into $\widehat{V}$ under the map $\xi_n$. In other words, $\widehat{D}=\xi_{n}^{-1}(\widehat{V})\subset D$.

\medskip
Now we claim that parameters in $\widehat{D}$ are hyperbolic parameters. This gives the desired approximation of non-recurrent parameters by hyperbolic ones, since $\widehat{D}\subset D'\subset D(\lambda_0, r)$ and $r$ can be chosen arbitrarily small.

\smallskip
For any $\lambda\in \widehat{D}$,  we see that $f_{\lambda}^{n+N+p}(0)\in D_{p+1}$. Now we consider a disk $B:=D(f_{\lambda}^{n+N+p}(0),1)$ centered at $f_{\lambda}^{n+N+p}(0)$ of radius $1$. By Koebe's one-quarter theorem we have
$$f_{\lambda}^{-(n+N+p)}\left(B\right)\supset D\left(0,\frac{1}{8\left|Df_{\lambda}^{n+N+p}(0)\right|}\right),$$
where we choose the inverse branch sending $f_{\lambda}^{n+N+p}(0)$ to $0$. 
On the other hand, the image $f_{\lambda}(B)$ lies in some left half-plane by the choice of $z_{p+1}$; see \eqref{zp+1}. More precisely,
$$f_{\lambda}(B)\subset\left\{z: \re(z)\leq - |\lambda|e^{\re (z_{p+1})-1}\right\}.$$
With
$$r_1:=\frac{1}{8\left|Df_{\lambda}^{n+N+p}(0)\right|}\quad\text{and}\quad r_2:=\exp\left\{- |\lambda|e^{\re (z_{p+1})-1}\right\},$$
we want to show that $r_1> r_2$, which means that a small disk around the singular value $0$ is mapped by $f_{\lambda}^{n+N+p+2}$ into itself. In other words, $f_{\lambda}$ has an attracting cycle and thus is a hyperbolic map. Noting that
$$
\begin{aligned}
\left|Df_{\lambda}^{n+N+p}(0)\right|=\prod_{j=0}^{n+N+p-1}\left|f'_{\lambda}(f_{\lambda}^{j}(0))\right|&=\prod_{j=0}^{n+N+p-1}\left|\lambda e^{\re f_{\lambda}^{j}(0) }\right|\\
&\leq |\lambda|^{n+N+p}e^{(n+N+p)\re (z_p)},
\end{aligned}
$$
one can obtain the following estimate for $r_1$:
$$r_1  \geq \frac{1}{8|\lambda|^{n+N+p}e^{(n+N+p)\re (z_p)}}.$$
On the other hand, since $\re(z_{p+1})=c'' |z_{p+1}|=c'' |\lambda|e^{\re(z_p)}$ for a small constant $c''$, we get that
$$r_2\leq\frac{1}{e^{|\lambda|\exp(c''|\lambda|\exp(\re(z_p)-1))}}.$$
This means that $r_2$ is much smaller than $r_1$. So we see that there is a neighborhood $U$ of $0$ which is mapped by $f_{\lambda}^{n+N+p+2}$ into itself. Therefore, $f_{\lambda}$ is a hyperbolic map.

This completes the proof of Theorem \ref{maintheorem2}.

\medskip
\begin{ack}
We would like to thank Lasse Rempe for useful comments and remarks, and Mats Bylund for some corrections. The second author was partially supported by Vergstiftelsen. We thank the referee for a detailed reading of the manuscript and valuable comments and suggestions.
\end{ack}

\bigskip

\noindent \emph{Magnus Aspenberg}: Centre for Mathematical Sciences, Lund University, Box 118, 22 100 Lund, Sweden
 
\smallskip
{magnus.aspenberg@math.lth.se}

\bigskip
\noindent \emph{Weiwei Cui}: 
Centre for Mathematical Sciences, Lund University, Box 118, 22 100 Lund, Sweden

\smallskip
{weiwei.cui@math.lth.se}
\end{document}